\documentclass[12pt, reqno]{amsart}
\usepackage{amsmath, amsthm, amscd, amsfonts, amssymb, graphicx, color}
\usepackage[bookmarksnumbered, colorlinks, plainpages]{hyperref}
\hypersetup{colorlinks=true,linkcolor=red, anchorcolor=green, citecolor=cyan, urlcolor=red, filecolor=magenta, pdftoolbar=true}

\newtheorem{theorem}{Theorem}[section]
\newtheorem{lemma}[theorem]{Lemma}
\newtheorem{proposition}[theorem]{Proposition}
\newtheorem{cor}[theorem]{Corollary}

\newtheorem{example}[theorem]{Example}
\newtheorem{question}[theorem]{Question}

\theoremstyle{remark}
\newtheorem{remark}[theorem]{\bf{Remark}}
\numberwithin{equation}{section}
\allowdisplaybreaks

\begin{document}
	\title [Berezin range of Toeplitz and composition operators]    
	{{On the Berezin range of Toeplitz and weighted composition operators on weighted Bergman spaces}}

\author[A. Sen, S. Barik and K. Paul]{Anirban Sen, Somdatta Barik and Kallol Paul}
	
\address[Sen] {Department of Mathematics, Jadavpur University, Kolkata 700032, West Bengal, India}
\email{anirbansenfulia@gmail.com}

\address[Barik]{Department of Mathematics, Jadavpur University, Kolkata 700032, West Bengal, India}
\email{bariksomdatta97@gmail.com}

\address[Paul] {Vice-Chancellor\\
	Kalyani University\\
	West Bengal 741235 \\and 
	Professor (on lien)\\ Department of mathematics\\ Jadavpur University\\Kolkata 700032\\West Bengal\\India}
\email{kalloldada@gmail.com}



\maketitle	
\begin{abstract}
    In this article, we completely characterize the Berezin range of Toeplitz operators with harmonic symbols acting on weighted Bergman spaces, illustrating the necessity of the harmonicity condition through examples. We then introduce a new class of weighted composition operators on these spaces, investigating their fundamental properties and determining their Berezin range and Berezin number. Finally, we study the convexity of the Berezin range of composition operators on weighted Bergman spaces and show that the origin lies in its closure of Berezin range but not in the range itself.
\end{abstract}

\section{Introduction}

In a reproducing kernel Hilbert space, the Berezin range of an
operator is contained within the numerical range of that operator. A fundamental and well-known property of the numerical range of an operator is its convexity. This naturally leads to the question:
Is the Berezin range of an operator also convex? In general, the Berezin range is not necessarily convex. This question was first raised by Karaev in \cite{K_CAOT_2013}. Subsequently, Cowen and Felder provided answers for certain class of composition operators in the Hardy-Hilbert space. 
In this article, our aim is to study the Berezin range as well as its convexity for two important classes of operators: Toeplitz operators and weighted composition operators.
At this point, we introduce the notation and key concepts that will serve as the foundation for the discussions to follow.

	Let $\mathbb{D}=\{w \in \mathbb{C} : |w|<1\}$ be the open unit disk of $\mathbb{C}$ and $\mathbb{T}=\{w \in \mathbb{C} : |w|=1\}$ be the unit circle of $\mathbb{C}.$ Let $dA$ denote the area measure on $\mathbb{D},$ normalized so that the area of $\mathbb{D}$ is 1. For $\gamma>-1,$ we consider the positive Borel measure $dA_{\gamma}(w)=(\gamma+1)(1-|w|^2)^{\gamma}dA(w)$ on $\mathbb{D}.$ For fixed $\gamma>-1,$ the weighted Bergman space is denoted by $A_{\gamma}^2(\mathbb{D}),$ and is defined as
	$$A^2_{\gamma}(\mathbb{D})=\left\{\text{$f \in \mathcal H(\mathbb{D})$} : \int_{\mathbb{D}}|f(w)|^2dA_{\gamma}<\infty \right\},$$
     where $\mathcal H(\mathbb{D})$ is the class of all holomorophic functions on $\mathbb{D}.$
	It is well known that $A^2_{\gamma}(\mathbb{D})$ is a reproducing kernel Hilbert space. The reproducing kernels and normalized reproducing kernels of $A^2_{\gamma}(\mathbb{D})$ are given by 
	$$k^{\gamma}_\xi(w)=\frac{1}{(1-\bar{\xi}w)^{\gamma+2}}\,\,\,\,\, \xi,w \in \mathbb{D}$$ 
	and
	$$\hat{k}^{\gamma}_\xi(w)=\frac{(1-|\xi|^2)^{\frac{\gamma+2}{2}}}{(1-\bar{\xi}w)^{\gamma+2}}\,\,\,\,\, \xi,w \in \mathbb{D},$$
	respectively. 
For more about weighted Bergman spaces, we refer to \cite{Zhu_book}.

Next, we introduce two important classes of operators on weighted Bergman spaces, namely, Toeplitz and weighted composition operators. 
Let $P_{\gamma}$ denote the orthogonal projection of $L^2(\mathbb{D}, dA_{\gamma})$ onto $A_\gamma^2(\mathbb D).$ Let $L^{\infty}(\mathbb{D}, dA_{\gamma})$ be the space of all complex measurable functions $\phi$ on $\mathbb{D}$ such that 
 $$\|\phi\|_{\infty,\gamma}=\sup\{ c\geq 0 : A_{\gamma}\left(\{z \in \mathbb{D} : |\phi(z)|>c\}\right)>0\}<\infty.$$
 For $\phi \in L^{\infty}(\mathbb{D}, dA_{\gamma}),$ the operator $T_{\phi}$ on $A_\gamma^2(\mathbb D)$ defined by
 $$T_{\phi}f=P_{\gamma}(\phi f)\,\,\,\, f \in A_\gamma^2(\mathbb D),$$
 is called the Toeplitz operator on $A_\gamma^2(\mathbb D)$ with symbol $\phi.$ It is easy to observe that $T_{\phi}$ is a bounded linear operator on $L_a^2(dA_{\gamma})$ with $\|T_{\phi}\|\leq \|\phi\|_{\infty,\gamma},$ see \cite{Zhu_book}.
 
Let $\phi : \mathbb{D} \to \mathbb{D}$ be an analytic self map and $\psi \in \mathcal H(\mathbb{D}).$ The weighted composition operator $C_{\psi,\phi } : \mathcal H(\mathbb{D}) \to \mathcal H(\mathbb{D})$ is defined by
 $$C_{\psi,\phi }f=\psi(f\circ \phi)\,\,\,\,\,f \in \mathcal H(\mathbb{D}).$$

 In particular, for $\psi=1,$ $C_{\psi,\phi }$ becomes the unweighted composition operator $C_{\phi}.$ In this article we limit our analysis of weighted composition operators on $A_\gamma^2(\mathbb D).$ There is a vast literature dealing with boundedness of weighted composition operators on weighted Bergman spaces, and we refer to \cite{CZ_IJM_2007,GKP_IEOT_2010}.

	For any bounded linear operator $T$ on $A^2_{\gamma}(\mathbb{D}),$ the Berezin transform of $T$ on $A^2_{\gamma}(\mathbb{D})$ is the function $\widetilde{T}: \mathbb{D} \to \mathbb{C}$ defined as
	\[\widetilde{T}(\xi)=\langle T\hat{k}^{\gamma}_\xi,\hat{k}^{\gamma}_\xi \rangle\,\,\,\,\xi \in \mathbb{D}.\]

It is well known that the Berezin transform is a one-to-one, bounded, and real-analytic function on $\mathbb D,$ see \cite{BER1972}. The properties of the operator are dependent on the Berezin transform of that operator. For instance, the Berezin transform uniquely determines the operator, and the invertibility of an operator is also dependent on its Berezin transform, see \cite{KAR_JFA_2006, ZZ_JOT_2016}. For more about Berezin transform we refer to \cite{E_JFA_1994,NR_Conference_1994, Zhu_AMS_2021, Z_PAMS_2003}.

	The Berezin range and Berezin number of $T$, denoted by $\textbf{Ber}(T)$ and  $\textbf{ber}(T),$ are respectively, defined as 
	\[\textbf{Ber}(T)=\{\widetilde{T}(\xi): \xi \in \mathbb{D}\} \,\, 
	and \,\, \textbf{ber}(T)=\sup_{\xi \in \mathbb{D}}|\widetilde{T}(\xi)|.\]

In \cite{K_CAOT_2013, KAR_JFA_2006}, Karaev introduced the Berezin range and Berezin radius, and determined the Berezin range of some classes of operators acting on Model space. Recently, in \cite{CC_LAA_2022} Cowen and Felder studied the convexity of the Berezin range of some classes of composition operators on the Hardy-Hilbert space. They raised a question about the convexity of the Toeplitz operator on general reproducing kernel Hilbert space and it remains unsolved till now.

In \cite{E_LAA_1995}, Engliš obtain a relation between the Berezin transform of Toeplitz operator on Hardy-Hilbert space and the Poission integral of its symbol, namely,
if $\phi\in L^{\infty}(\mathbb{T})$ then $\widetilde{T_{\phi}}(\xi)=\mathcal P[\phi](\xi),$ where $\mathcal P[\phi]$ is the Poission integral of $\phi$ on $\mathbb D.$ Hence $\textbf{Ber}(T_{\phi})=\mathcal P[\phi](\mathbb{D})$ and so the Berezin range of $T_{\phi}$ is convex if and only if $\mathcal P[\phi](\mathbb{D})$ is convex. However, verifying the convexity of the range of the Poisson integral of $\phi$ remains a nontrivial task. Nevertheless, we can conclude that the Berezin range of $T_{\phi}$ is always a connected subset of $\mathbb C.$

The convexity of composition operators on the Bergman space, the Fock space, and the Dirichlet space has also been studied in \cite{AGS_CAOT_2024, AGS_CAOT_2023}. Apart from these, there are only a few articles on the geometry of the Berezin range of operators on functional Hilbert spaces, and the study of the Berezin range has yet to be explored in its full depth.

Including the introductory one, our article is divided into four sections. In Section 2, we completely characterize the Berezin range of Toeplitz operator with harmonic symbol acting on weighted Bergman spaces. We also provide examples to demonstrate that the harmonic property is necessary for this characterization. Our main aim in Section 3 is to introduce a new class of weighted composition operators on weighted Bergman spaces. We study some basic properties and obtain Berezin range and Berezin number of these operators. In Section 4, we characterize the convexity of the Berezin range of composition operators on weighted Bergman spaces. Also, we obtain that the origin is contained in the closure of the Berezin range of composition operator but not in the Berezin range of that operator.

	\section{Berezin range of Toeplitz operator}

This section is devoted to the complete determination of the Berezin range associated with Toeplitz operators $T_\phi$ acting on 
$A^2_{\gamma}(\mathbb{D})$ with harmonic symbol $\phi$. The analysis begins with the following key theorem.

	\begin{theorem}\label{WBT1}
		If $\phi \in L^{\infty}(\mathbb{D},dA_{\gamma})$ is harmonic on $\mathbb{D}$, then $\textbf{Ber}(T_{\phi})=\phi(\mathbb{D})$.
	\end{theorem}
	
	\begin{proof}
		Let $\hat{k}^{\gamma}_w$ be the normalized reproducing kernel at $w \in \mathbb{D}.$ Then from the definition of the Berezin transform we have
		\begin{align*}
			\widetilde{T_{\phi}}(w)=\langle T_{\phi}\hat{k}^{\gamma}_w,\hat{k}^{\gamma}_w\rangle
			=\langle P_{\gamma}(\phi\hat{k}^{\gamma}_w),\hat{k}^{\gamma}_w\rangle
			=\langle \phi\hat{k}^{\gamma}_w,P_{\gamma}\hat{k}^{\gamma}_w\rangle
			=\langle \phi\hat{k}^{\gamma}_w,\hat{k}^{\gamma}_w\rangle.
		\end{align*}
		This implies that
		\begin{align*}
			\widetilde{T_{\phi}}(w) =\int_{\mathbb{D}}\phi(z)|\hat{k}^{\gamma}_w(z)|^2dA_{\gamma}(z)
			&=\int_{\mathbb{D}}\phi(z)\frac{(1-|w|^2)^{\gamma+2}}{|1-\bar{w}z|^{2(\gamma+2)}}dA_{\gamma}(z)\\
			&=\int_{\mathbb{D}}\phi \circ \phi_{w}(z)dA_{\gamma}(z),
		\end{align*}
		where $\phi_{w}(z)=\frac{w-z}{1-\bar{w}z}$ for all $z \in \mathbb{D}.$
		Now, by the mean-value property of harmonic functions (see \cite{AFW_JFA_1993}), we obtain
		\begin{align*}
			\widetilde{T_{\phi}}(w)
			=\int_{\mathbb{D}}\phi \circ \phi_{w}(z)dA_{\gamma}(z)=\phi \circ \phi_{w}(0)=\phi(w).
		\end{align*}
		Therefore, we obtain the desired relation $\textbf{Ber}(T_{\phi})=\phi(\mathbb{D})$.
	\end{proof}

	It is straightforward to establish the convexity of the Toeplitz operator $T_\phi$ based on the previous theorem.
	
	\begin{cor}
		Let $\phi \in L^{\infty}(\mathbb{D},dA_{\gamma})$ be harmonic on $\mathbb{D}$. Then the Berezin range of $T_{\phi}$ is convex if and only if $\phi(\mathbb{D})$ is convex.
	\end{cor}

Next, we provide examples to demonstrate that the harmonic property of $\phi$ is necessary for the characterization of the Berezin range of the Toeplitz operator $T_\phi.$

	\begin{example}
		(i)~~Consider the function
		$$\phi(w)=\begin{cases}
			1,~~~\,\,\,\, |w| \leq \frac{1}{2}\\
			0,~~~\,\,\,\, \frac{1}{2} <|w| < 1.
		\end{cases}$$
		Clearly, $\phi$ is not continuous on $\mathbb{D}$ and $\phi(\mathbb{D})=\{0,1\}.$ Now, for $w \in \mathbb{D},$ we have
		\begin{align*}
			\widetilde{T_{\phi}}(w)
			=\langle T_{\phi}\hat{k}^{\gamma}_w,\hat{k}^{\gamma}_w\rangle
			=\int_{\mathbb{D}}\phi(z)|\hat{k}^{\gamma}_w(z)|^2dA_{\gamma}(z)
			=\int_{\mathbb{D}_{1/2}}|\hat{k}^{\gamma}_w(z)|^2dA_{\gamma}(z),
		\end{align*}
		where $\mathbb{D}_{1/2}=\{z \in \mathbb{D} : |z| \leq \frac12\}.$ 
		Thus if $\widetilde{T_{\phi}}(w)=0$ then we have $$\int_{\mathbb{D}_{1/2}}|\hat{k}^{\gamma}_w(z)|^2dA_{\gamma}(z)=0.$$ 
		This implies that $\hat{k}^{\gamma}_w(z)=0$ a.e. on $\mathbb{D}_{1/2},$ which is an absurd. Similarly, if $\widetilde{T_{\phi}}(w)=1$ then we have 
		$$\int_{\mathbb{D} \setminus\mathbb{D}_{1/2}}|\hat{k}^{\gamma}_w(z)|^2dA_{\gamma}(z)=0$$
		and this implies that $\hat{k}^{\gamma}_w(z)=0$ a.e. on $\mathbb{D} \setminus\mathbb{D}_{1/2},$ which is an absurd.
		Therefore, for this example, $\textbf{Ber}(T_{\phi})$ and $\phi(\mathbb{D})$ are disjoint subsets of $\mathbb{C}.$ Thus the condition $\phi$ is harmonic on $\mathbb{D}$ is necessary in Theorem \ref{WBT1}.

	(ii)~~Consider the function $\phi(w)=|w|^2$ on $\mathbb{D}.$ Then $\phi$ is continuous on $\mathbb{D}$ but the Laplacian of $\phi$ is non-zero on $\mathbb{D}.$ For $w \in \mathbb{D},$ we have
		\begin{align*}
			\widetilde{T_{\phi}}(w)
			=\langle T_{\phi}\hat{k}^{\gamma}_w,\hat{k}^{\gamma}_w\rangle
			=\langle \phi\hat{k}^{\gamma}_w,\hat{k}^{\gamma}_w\rangle
			=(1-|w|^2)^{\gamma+2}\int_{\mathbb{D}}|z\hat{k}^{\gamma}_w(z)|^2dA_{\gamma}(z).
		\end{align*}
		Now, by a simple computation we obtain
		\begin{align*}
			(1-|w|^2)^{\gamma+2}\int_{\mathbb{D}}|z\hat{k}^{\gamma}_w(z)|^2dA_{\gamma}(z)
			=(1-|w|^2)^{\gamma+2}\sum_{n=0}^{\infty}\frac{(n+1)\Gamma(n+\gamma+2)}{n!(n+\gamma+2)\Gamma(\gamma+2)}|w|^{2n}.
		\end{align*}
		Since, $\frac{n+1}{n+\gamma+2}\geq \frac{1}{\gamma+2}$ for all $n \geq 0$ so we get
		\begin{align*}
			&(1-|w|^2)^{\gamma+2}\sum_{n=0}^{\infty}\frac{(n+1)\Gamma(n+\gamma+2)}{n!(n+\gamma+2)\Gamma(\gamma+2)}|w|^{2n}\\
			&\geq \frac{1}{\gamma+2}(1-|w|^2)^{\gamma+2}\sum_{n=0}^{\infty}\frac{\Gamma(n+\gamma+2)}{n!\Gamma(\gamma+2)}|w|^{2n} = \frac{1}{\gamma+2}.
		\end{align*}
		Again, it is easy to observe that $\widetilde{T_{\phi}}(w)<1.$ Therefore, for this example
		$$\textbf{Ber}(T_{\phi}) \subseteq \Big[\frac{1}{\gamma+2},1\Big) \subsetneq (0,1)=\phi(\mathbb{D}).$$
		Thus the condition that $\phi$ is harmonic on $\mathbb{D}$ is necessary in Theorem \ref{WBT1}.
	\end{example}

From the first example, we observe that the Berezin range of the Toeplitz operator and the range of its symbol are disjoint sets, whereas in the second example, the Berezin range is properly contained within the range of its symbol. This leads to the following open question: 
\begin{question}
Can the Berezin range be described for Toeplitz operators with non-harmonic or more general symbols?
\end{question}

We now state several corollaries that follow from Theorem \ref{WBT1}.
    To proceed, recall that the numerical range of a bounded linear operator $T$ on $A_\gamma^2(\mathbb D)$ is defined as 
    $$W(T)=\{\langle Tf, f\rangle: f\in A_\gamma^2(\mathbb D), \|f\|=1\}.$$ Further details on the numerical range can be found in the book \cite{Gustafson_book}.

To derive the following corollary, we first require the following lemma, which is recently proved in \cite[Th. 3.6]{Sen_2025}.
    
	\begin{lemma}\label{lm_121}
		Let $\phi \in L^{\infty}(\mathbb{D}, dA_{\gamma})$ be a non-constant harmonic function on $\mathbb{D},$ then $W(T_{\phi})=$ $Rel$ $int$ $\overline{\phi(\mathbb{D})}^{\wedge}.$
	\end{lemma}
	
Using Lemma \ref{lm_121}, the following corollary follows directly from Theorem \ref{WBT1}, which establishes a clear relationship between the numerical range and the Berezin range of Toeplitz operators with harmonic symbol.
	\begin{cor}
		Let $\phi \in L^{\infty}(\mathbb{D}, dA_{\gamma})$ be a non-constant harmonic function on $\mathbb{D},$ then $W(T_{\phi})=$ $Rel$ $int$ $\overline{\textbf{Ber}(T_{\phi})}^{\wedge}.$
	\end{cor}

    The following result characterizes normal Toeplitz operators with harmonic symbol in terms of their Berezin range, and it follows directly from Theorem \ref{WBT1} together with the next lemma.
	
	\begin{lemma}\cite[Cor. 7.28]{Zhu_book}\label{lk_008}
		Let $\phi$ be a bounded harmonic function on $\mathbb{D}.$ Then $T_{\phi}$ is normal on $A^2_{\gamma}(\mathbb{D})$ if and only if $\phi(\mathbb{D})$ lies on a straight line in $\mathbb{C}.$
	\end{lemma}

	\begin{cor}
		Let $\phi$ be a bounded harmonic function on $\mathbb{D}.$ Then $T_{\phi}$ is normal on $A^2_{\gamma}(\mathbb{D})$ if and only if $\textbf{Ber}(T_{\phi})$ lies on a straight line in $\mathbb{C}.$
	\end{cor}

Next, we show that for any non-empty, bounded, and simply connected region $G$, there exists a corresponding Toeplitz operator whose Berezin range is equal to $G.$

\begin{cor}
    Let $G$ be a non-empty, bounded and simply connected region in $\mathbb C$. Then for any $\gamma>-1$ there exists a Toeplitz operator $T_\phi$ on $A_\gamma^2(\mathbb D)$ such that 
    $\textbf{Ber}(T_\phi)=G.$
\end{cor}

\begin{proof}
    Since $G$ is a non-empty, bounded and simply connected region in $\mathbb C,$ then by Riemann mapping theorem there exists a conformal map $\phi$ with $\phi(\mathbb D)=G.$ As $\phi$ is bounded on $\mathbb D$ so by Theorem \ref{WBT1}, we have $\textbf{Ber}(T_\phi)=\phi(\mathbb D)=G.$
\end{proof}


\section{Weyl-type operator on weighted Bergman spaces}

	The study of isometric and unitary weighted composition operators on reproducing kernel Hilbert spaces of holomorphic functions has received considerable attention over the years.
	In this section, we introduce a new class of operators in the space $A^2_{\gamma}(\mathbb{D}),$ which extends the concept of Weyl-type unitary operators previously considered in \cite{C_PAMS_2012}. 
Let $\psi_{\beta, \eta}$ be the automorphism of the unit disc $\mathbb D$ defined by $\psi_{\beta, \eta}(w)=\eta\frac{w-\beta}{1-\bar{\beta}w},$ where $\eta\in\mathbb T$ and $\beta\in\mathbb D.$ 
	We introduce the operator $C_{\hat k_\beta^{\gamma}, \psi_{\beta, \eta}}$ on $A^2_{\gamma}(\mathbb{D})$ defined as $$C_{\hat k_\beta^{\gamma}, \psi_{\beta, \eta}}(f)=\hat k_\beta^{\gamma}(f\circ\psi_{\beta, \eta})\,\,\,\,\forall f\in A^2_{\gamma}(\mathbb{D}).$$
    
	We begin by proving that the operator $C_{\hat k_\beta^{\gamma}, \psi_{\beta, \eta}}$ is unitary on the space $A^2_{\gamma}(\mathbb{D}).$
	\begin{proposition}
		$C_{\hat k_\beta^{\gamma}, \psi_{\beta, \eta}}$ is an unitary operator on $A^2_{\gamma}(\mathbb{D}).$
		
	\end{proposition}	
	\begin{proof}
		First we prove that the operator $C_{\hat k_\beta^{\gamma}, \psi_{\beta, \eta}}$ is an isometry on $A^2_{\gamma}(\mathbb{D}).$
		\begin{eqnarray*}
			\|C_{\hat k_\beta^{\gamma}, \psi_{\beta, \eta}} f\|^2&=&\int_{\mathbb D}|\hat k_\beta^{\gamma}(w)|^2|f(\psi_{\beta, \eta}(w))|^2dA_{\gamma}(w)\\
			&=& \int_{\mathbb D}|f(\psi_{\beta, \eta}(w))|^2\frac{(1-|\beta|^2)^{2+\gamma}}{|1-\bar{\beta}w|^{2(\gamma+2)}}dA_{\gamma}(w)\\
			&=& \int_{\mathbb D}|f(w)|^2dA_{\gamma}(w)=\|f\|^2.
		\end{eqnarray*}
		Now, we have to show that $C_{\hat k_\beta^{\gamma}, \psi_{\beta, \eta}}$ is onto, i.e., for any $g\in A^2_{\gamma}(\mathbb{D})$ there exists $f\in A^2_{\gamma}(\mathbb{D})$ such that $C_{\hat k_\beta^{\gamma}, \psi_{\beta, \eta}} f=g.$ For $w\in\mathbb D,$ if we consider $f(w)=\frac{g(\psi_{\beta, \eta}(w))}{\hat k_\beta^{\gamma}(\psi_{\beta, \eta}(w))}$ then it is easy to observe that $f\in A^2_{\gamma}(\mathbb{D})$ and satisfies $C_{\hat k_\beta^{\gamma}, \psi_{\beta, \eta}} f=g.$ This completes the proof.

	\end{proof}

	The Berezin transform of $C_{\hat k_\beta^{\gamma}, \psi_{\beta, \eta}}$ on $A^2_{\gamma}(\mathbb{D})$ is given by 
	
	\begin{eqnarray}
		\widetilde{C_{\hat k_\beta^{\gamma}, \psi_{\beta, \eta}}}(\xi)=\langle C_{\hat k_\beta^{\gamma}, \psi_{\beta, \eta}}\hat k_{\xi}^{\gamma}, \hat k_{\xi}^{\gamma} \rangle
		&=&(1-|\beta|^2)^{\frac{\gamma+2}{2}}(1-|\xi|^2)^{\gamma+2} k_{\beta}^{\gamma}(\xi) k_{\xi}^{\gamma}(\psi_{\beta, \eta}(\xi))\nonumber\\
		&=&\frac{(1-|\beta|^2)^{\frac{\gamma+2}{2}}(1-|\xi|^2)^{\gamma+2}}{(1-\bar{\beta}\xi)^{\gamma+2}(1-\bar{\xi}\psi_{\beta, \eta}(\xi))^{\gamma+2}}\nonumber\\
		&=&\left(\frac{(1-|\beta|^2)(1-|\xi|^2)^2}{\left((1-\eta|\xi|^2)-(\xi\bar{\beta}-\eta\bar{\xi}\beta)\right)^2}\right)^{\frac{\gamma+2}{2}}.
	\end{eqnarray}

    Next, we verify that although $0$ lies in the closure of the Berezin range of $C_{\hat k_\beta^{\gamma}, \psi_{\beta, \eta}},$  it does not belong to the range itself.
	\begin{proposition}\label{prop_9}
            For the Weyl-type unitary operator $C_{\hat k_\beta^{\gamma}, \psi_{\beta, \eta}}$ on $A^2_{\gamma}(\mathbb{D}),$ 
            $$0\in\overline{\textbf{Ber}(C_{\hat k_\beta^{\gamma}, \psi_{\beta, \eta}})}\setminus\textbf{Ber}(C_{\hat k_\beta^{\gamma}, \psi_{\beta, \eta}}).$$
	\end{proposition}
	\begin{proof}
           For any $\xi\in \mathbb D,$ we have 
     \begin{eqnarray}\label{eq_422}
                   \widetilde{C_{\hat k_\beta^{\gamma}, \psi_{\beta, \eta}}}(\xi)
		&&=\frac{(1-|\beta|^2)^{\frac{\gamma+2}{2}}(1-|\xi|^2)^{\gamma+2}}{(1-\bar{\beta}\xi)^{\gamma+2}(1-\bar{\xi}\psi_{\beta, \eta}(\xi))^{\gamma+2}}.
    \end{eqnarray}

	    Since $\psi_{\beta, \eta}$ is not an identity map on $\mathbb{D}$, there exists $\xi_0 \in \mathbb{T}$ such that the radial limit of $\psi_{\beta, \eta}$ exists at $\xi_0$ with $\psi_{\beta, \eta}(\xi_0) \neq \xi_0.$ Therefore, from \eqref{eq_422} we get $\widetilde{C_{\hat k_\beta^{\gamma}, \psi_{\beta, \eta}}}(\xi) \to 0$ when  $\xi$ radially approaches to the point $\xi_0.$ Hence, we obtain $0 \in \overline{\textbf{Ber}(C_{\hat k_\beta^{\gamma}, \psi_{\beta, \eta}})}.$
		
		Again, from \eqref{eq_422} we get
		\begin{align*}
			\left|\widetilde{C_{\hat k_\beta^{\gamma}, \psi_{\beta, \eta}}}(\xi)\right|&=\frac{(1-|\beta|^2)^{\frac{\gamma+2}{2}}(1-|\xi|^2)^{\gamma+2}}{|1-\bar{\beta}\xi|^{\gamma+2}|1-\bar{\xi}\psi_{\beta, \eta}(\xi)|^{\gamma+2}}\\
            &\geq \frac{(1-|\beta|^2)^{\frac{\gamma+2}{2}}(1-|\xi|^2)^{\gamma+2}}{2^{2{(\gamma+2)}}}>0\,\,\text{for all $\xi \in \mathbb{D}$}.
		\end{align*}
		Therefore, we obtain $\widetilde{C_{\hat k_\beta^{\gamma}, \psi_{\beta, \eta}}}(\xi) \neq 0$ for all $\xi \in \mathbb{D},$ as desired.
	\end{proof}

    Now, we proceed to obtain the Berezin range of the operator $C_{\hat k_\beta^{\gamma}, \psi_{\beta, \eta}}$ on $A^2_{\gamma}(\mathbb{D})$ for $\eta=-1$.
     \begin{theorem}
         For the Weyl-type unitary operator $C_{\hat k_\beta^{\gamma}, \psi_{\beta,-1}}$ on $A^2_{\gamma}(\mathbb{D}),$ $\textbf{Ber}\left(C_{\hat k_\beta^{\gamma}, \psi_{\beta,-1}}\right)=(0,1].$
     \end{theorem}
     \begin{proof}
         For $\eta=-1,$ we have
		$$\langle C_{\hat k_\beta^{\gamma}, \psi_{\beta, -1}}\hat k_{\xi}^{\gamma}, \hat k_{\xi}^{\gamma} \rangle=\left(\frac{(1-|\beta|^2)(1-|\xi|^2)^2}{\left((1+|\xi|^2)-(\xi\bar{\beta}+\bar{\xi}\beta)\right)^2}\right)^{\frac{\gamma+2}{2}}.$$
		The equation $\psi_{\beta, -1}(\xi)=\xi$ has a unique solution in $\mathbb D, $ $\xi(\beta)=\frac{1-\sqrt{1-|\beta|^2}}{\bar{\beta}}.$ Thus, by simple computation we get 
        $$\widetilde{ C_{\hat k_\beta^{\gamma}, \psi_{\beta, -1}}}(\xi(\beta))=1.$$
        Since $\widetilde{ C_{\hat k_\beta^{\gamma}, \psi_{\beta, -1}}}$ is a continuous function, it follows that the range of $\widetilde{ C_{\hat k_\beta^{\gamma}, \psi_{\beta, -1}}}$ is a connected subset of $\mathbb R$ that contains 1 and $\widetilde{ C_{\hat k_\beta^{\gamma}, \psi_{\beta, -1}}}(\xi)\to 0$ as $|\xi|\to1.$ 
        Therefore, $\textbf{Ber}\left(C_{\hat k_\beta^{\gamma}, \psi_{\beta,-1}}\right)=(0,1]$.
     \end{proof}

	 Next, we evaluate the Berezin number of the operator $C_{\hat k_\beta^{\gamma}, \psi_{\beta, \eta}}$ on $A^2_{\gamma}(\mathbb{D})$ corresponding to $\eta=1.$
	\begin{theorem}\label{TT1}
		For the Weyl-type unitary operator $C_{\hat k_\beta^{\gamma}, \psi_{\beta, 1}}$ on $A^2_{\gamma}(\mathbb{D}),~~~~\textbf{ber}\left(C_{\hat k_\beta^{\gamma}, \psi_{\beta, 1}}\right)=(1-|\beta|^2)^{\frac{\gamma+2}{2}}.$ 
	\end{theorem}
	
	\begin{proof}
		We have
		$$\langle C_{\hat k_\beta^{\gamma}, \psi_{\beta, 1}}\hat k_{\xi}^{\gamma}, \hat k_{\xi}^{\gamma} \rangle=\left(\frac{(1-|\beta|^2)(1-|\xi|^2)^2}{\left((1-|\xi|^2)-(\xi\bar{\beta}-\bar{\xi}\beta)\right)^2}\right)^{\frac{\gamma+2}{2}}.$$
		Since $i(\xi\bar{\beta}-\bar{\xi}\beta)$ is real, we get
		$$|(1-|\xi|^2)-(\xi\bar{\beta}-\bar{\xi}\beta)|^2=(1-|\xi|^2)^2+|\xi\bar{\beta}-\bar{\xi}\beta|^2.$$
		So, 
		$$|\langle C_{\hat k_\beta^{\gamma}, \psi_{\beta, 1}}\hat k_{\xi}^{\gamma}, \hat k_{\xi}^{\gamma} \rangle|=\left(\frac{(1-|\beta|^2)(1-|\xi|^2)^2}{(1-|\xi|^2)^2+|\xi\bar{\beta}-\bar{\xi}\beta|^2}\right)^{\frac{\gamma+2}{2}}\leq (1-|\beta|^2)^{\frac{\gamma+2}{2}}. $$
		Since $\widetilde{ C_{\hat k_\beta^{\gamma}, \psi_{\beta, 1}}}(0)=(1-|\beta|^2)^{\frac{\gamma+2}{2}},$ it follows that $\textbf{ber}(C_{\hat k_\beta^{\gamma}, \psi_{\beta, 1}})=(1-|\beta|^2)^{\frac{\gamma+2}{2}}.$\\
		
	\end{proof}

	\begin{cor}
	    The space $\{\widetilde{T}:T\in\mathcal B(A^2_{\gamma}(\mathbb{D})), \|\cdot\|_{\infty}\}$ is a non-closed linear subspace of $\mathcal B \mathcal C (\mathbb D),$ where $\mathcal B \mathcal C (\mathbb D)$ is the space of all bounded continuous functions on $\mathbb D$ and $\mathcal B(A^2_{\gamma}(\mathbb{D}))$ is the space of all bounded linear operators on $A^2_{\gamma}(\mathbb{D}).$
	\end{cor}
	
	\begin{proof}
	    Since $C_{\hat k_\beta^{\gamma}, \psi_{\beta, 1}}$ is unitary, we have $\|C_{\hat k_\beta^{\gamma}, \psi_{\beta, 1}}\|=1.$ As $\beta\in\mathbb D$ is arbitrary, so Theorem \ref{TT1} implies that $\textbf{ber}\left(C_{\hat k_\beta^{\gamma}, \psi_{\beta, 1}}\right)=(1-|\beta|^2)^{\frac{\gamma+2}{2}}\to 0$ as $|\beta| \to 1$. Therefore, there does not exist any fixed $k>0$ such that $k\textbf{ber}(T)\geq\|T\|$ for all $T\in \mathcal B(A^2_{\gamma}(\mathbb{D})).$ \\
        We consider the function $\phi:\mathcal B(A^2_{\gamma}(\mathbb{D}))\to \mathcal B \mathcal C(\mathbb D)$ defined as $\phi(T)=\widetilde{T}.$ Clearly, $\phi$ is bounded and linear. Suppose that $\{\widetilde{T}:T\in\mathcal B(A^2_{\gamma}(\mathbb{D}))\}$ is closed in $\mathcal B \mathcal C(\mathbb D).$ Then $\phi$ is a one-to-one mapping on $\mathcal B \mathcal C(\mathbb D)$ . The open mapping theorem ensures that there exists a constant $k$ such that $\|T\|\leq k \textbf{ber}(T)$ for all $T\in \mathcal B(A^2_{\gamma}(\mathbb{D})).$ This contradicts the fact that there does not exist any $k>0$ such that $k\textbf{ber}(T)\geq\|T\|$ for all $T\in \mathcal B(A^2_{\gamma}(\mathbb{D})).$ Therefore, $\{\widetilde{T}:T\in\mathcal B(A^2_{\gamma}(\mathbb{D})), \|\cdot\|_{\infty}\}$ is a non-closed linear subspace of $\mathcal B \mathcal C (\mathbb D).$
	\end{proof}

	\section{Convexity of Berezin range of composition operators} 
	
	The conformal mappings $\psi_{\beta, \eta}:\mathbb D\to\mathbb D$ are the functions of the form $\psi_{\beta, \eta}(w)=\eta\frac{w-\beta}{1-\bar{\beta}w},$ where $\eta\in\mathbb T.$ In this section, we study the Berezin ranges of composition operators $C_{\psi_{\beta, \eta}}$ on $A^2_{\gamma}(\mathbb{D})$  induced by $\psi_{\beta, \eta}$ on $\mathbb D.$ In particular, we analyze the convexity of Berezin range of the composition operator $C_{\psi_{\beta, \eta}}$ on $A^2_{\gamma}(\mathbb{D})$ induced by $\psi_{\beta, \eta},$ where $\psi_{\beta, \eta}$ is of elliptic type, and later when $\psi_{\beta, \eta}$ is a Blashke factor.

	The Berezin transform of $C_{\psi_{\beta, \eta}}$ on $A^2_{\gamma}(\mathbb{D})$ is given by
	\begin{align}\label{e1}
		\widetilde{C_{\psi_{\beta, \eta}}}(w)
        =\frac{1}{\|k^{\gamma}_w\|^2}\langle C_{\psi_{\beta, \eta}}{k}^{\gamma}_w, {k}^{\gamma}_w \rangle
        =\frac{1}{\|k^{\gamma}_w\|^2}{k}^{\gamma}_w(\psi_{\beta, \eta}(w))
        =\left(\frac{1-|w|^2}{1-\bar{w}\psi_{\beta, \eta}(w)}\right)^{\gamma+2}.
	\end{align}

	We begin by presenting a characterization of the convexity of the Berezin range of the composition operator $C_{\psi_{0, \eta}},$ where $\psi_{0, \eta}$ is of elliptic type. To this end, we first establish a lemma that will be used in the proof of the main theorem. 
	\subsection{Elliptic Symbol}
	Let $\psi_{0, \eta}(w)=\eta w$ be the elliptic automorphism on $\mathbb D,$  where $\eta\in\mathbb T$ and $w\in\mathbb D.$ Then 
	\begin{equation*}
		\widetilde{C_{{\psi_{0, \eta}}}}(w)=\langle C_{{\psi_{0, \eta}}}\hat{k}^{\gamma}_w(w), \hat{k}^{\gamma}_w(w)\rangle
		=\left(\frac{1-|w|^2}{1-\eta |w|^2}\right)^{\gamma+2},
	\end{equation*}
	where $\gamma>-1$.

	\begin{lemma}\label{lm_1}
		Let ${\psi_{0, \eta}}(w)=\eta w$ with $\eta\in\mathbb T$ and $w\in\mathbb D$. Then 
		\begin{align*}
			&(i)~~\textbf{Ber}(C_{{\psi_{0, \eta}}})~~\text{is singleton if and only if}~~\eta=1,\\
			&(ii)~~\Im\{\textbf{Ber}(C_{{\psi_{0, \eta}}})\}=\{0\}~~\text{if and only if}~~\Im\{\eta\}=0.
		\end{align*}
	\end{lemma}

	\begin{proof}
		$(i)$ As the sufficient part is obvious, we only prove the necessary part.
		Let $w=re^{i\theta}$ with $r\in[0, 1)$ and and $\theta\in[0,2\pi).$ Then 
		\begin{equation*}
			\widetilde{C_{{\psi_{0, \eta}}}}(re^{i\theta})=\left(\frac{1-r^2}{1-\eta r^2}\right)^{\gamma+2}.
		\end{equation*}
		Suppose that $\textbf{Ber}(C_{{\psi_{0, \eta}}})$ is singleton. As $\widetilde{C_{{\psi_{0, \eta}}}}(0)=1,$ so $\textbf{Ber}(C_{{\psi_{0, \eta}}})=\{1\}.$
		Thus, $\left(\frac{1-r^2}{1-\eta r^2}\right)^{\gamma+2}=1$ for all $r\in[0, 1).$
		This implies that 
		\begin{equation}\label{eq2}
			\left|\frac{1-r^2}{1-\eta r^2}\right|=1\implies 1-r^2=|1-\eta r^2|.
		\end{equation}
		Now consider $\eta=e^{i\phi}$ in \eqref{eq2}, then
		\begin{equation*}
			1-r^2=|1- e^{i\phi}r^2|=|1- (\cos\phi+i~\sin \phi)r^2|=|(1-r^2\cos \phi)-i~r^2 \sin \phi|.
		\end{equation*}
		This implies that 
		\begin{equation*}
			2r^2(\cos^2 \phi-1)=0\implies \cos\phi=\pm 1\implies \sin \phi=0.
		\end{equation*}
		Therefore, $\eta=\cos\phi+i \sin\phi= 1,$ as $\eta=-1$ does not satisfy the equation \eqref{eq2} for all $r\in[0, 1).$

		$(ii)$ Suppose that $\Im\{\textbf{Ber}(C_{{\psi_{0, \eta}}})\}=0$. Then 
		$\frac{1-r^2}{1-\eta r^2}\in\mathbb R$ for all $r\in[0, 1).$ Let $\frac{1-r^2}{1-\eta r^2}=k_r,$ where $k_r\in\mathbb R.$ Clearly, $k_r\neq 0$ for all $r\in[0, 1).$ Then
		$\eta=\frac{1}{r^2}-\frac{1-r^2}{k_r r^2}\in\mathbb R.$ Thus $\Im\{\eta\}=0.$ The converse part is obvious.
	\end{proof}

	At this point, we can identify the values of 
$\eta$ that ensure the Berezin range of $C_{\psi_{0, \eta}}$ is convex.
	\begin{theorem}\label{th1}
		Let $C_{{\psi_{0, \eta}}}\in\mathcal B(A^2_{\gamma}(\mathbb{D}))$ be  such that ${\psi_{0, \eta}}(w)=\eta w$ with $\eta\in\mathbb T$ and $w\in\mathbb D.$ Then $\textbf{Ber}(C_{{\psi_{0, \eta}}})$ is convex if and only if $\eta=1$ or $-1$.
	\end{theorem}

	\begin{proof}
		Let $w=re^{i\theta}$ with $r\in[0, 1)$ and $\theta\in[0,2\pi).$ Then 
		\begin{equation*}
			\widetilde{C_{{\psi_{0, \eta}}}}(re^{i\theta})=\left(\frac{1-r^2}{1-\eta r^2}\right)^{\gamma+2}.
		\end{equation*}
		If $\eta=1$ then it follows from Lemma \ref{lm_1} $(i)$ that $\textbf{Ber}(C_{{\psi_{0, \eta}}})$ is singleton and so $\textbf{Ber}(C_{{\psi_{0, \eta}}})$ is convex.
		If $\eta=-1$ then it is easy to observe that $\textbf{Ber}(C_{{\psi_{0, \eta}}})=(0, 1],$ which is again convex.\\
		Conversely, let $\textbf{Ber}(C_{{\psi_{0, \eta}}})$ be convex. now, 
		\begin{equation*}
			\widetilde{C_{{\psi_{0, \eta}}}}(re^{i\theta})=\left(\frac{1-r^2}{1-\eta r^2}\right)^{\gamma+2},
		\end{equation*}
		which is independent of $\theta.$ Clearly, $\textbf{Ber}(C_{{\psi_{0, \eta}}})$ is a path in $\mathbb C.$ Now, the convexity of $\textbf{Ber}(C_{{\psi_{0, \eta}}})$ implies that $\textbf{Ber}(C_{{\psi_{0, \eta}}})$ is either a line segment or a point. If $\textbf{Ber}(C_{{\psi_{0, \eta}}})$ is a point then from Lemma \ref{lm_1} $(i)$, it follows that $\eta=1$. Next, consider $\textbf{Ber}(C_{{\psi_{0, \eta}}})$ is a line segment. Now, we have $\widetilde{C_{{\psi_{0, \eta}}}}(0)=1$ and $\lim\limits_{r\to 1^-}\widetilde{C_{{\psi_{0, \eta}}}}(re^{i\theta})=0.$ Thus we can conclude that $\textbf{Ber}(C_{{\psi_{0, \eta}}})$ is a line segment passing through the point 1 and approaching the origin and so $\Im\{\textbf{Ber}(C_{{\psi_{0, \eta}}})\}=\{0\}.$ Then by applying Lemma \ref{lm_1} $(ii)$, we have 
		$\Im\{\eta\}=0$. Therefore, $\eta=1$ or $\eta=-1$.
	\end{proof}

    As a direct implication of the previous theorem, we present the following remark.
	\begin{remark}
		Let $\psi(w)=\frac{cw+d}{\bar{d}w+\bar{c}},$ where $c, d\in\mathbb C$ and $|c|^2-|d|^2=1.$ Then $\widetilde{C_{\psi}}(w)=\left(\frac{1-|w|^2}{1-\bar{w}\psi(w)}\right)^{\gamma+2}.$ If $d=0$ then $\psi(w)=\frac{c}{\bar{c}}w.$ Using Theorem \ref{th1}, $\textbf{Ber}(C_{\psi})$ is convex if and only if $c=1, -1, i, -i.$
	\end{remark}

    We turn our attention to analyze the convexity of the Berezin range of $C_{\psi_{\beta, 1}},$ where $\psi_{\beta, 1}$ is Blaschke factor. To do so, we first present a preliminary computational observation. 
	\subsection{Blaschke factor symbol}
	
	Let $\psi_{\beta, 1}(w)=\frac{w-\beta}{1-\bar{\beta}w}$ where $\beta, w \in\mathbb D.$ Then 
	\begin{equation}\label{eq3}
		\widetilde{C_{\psi_{\beta, 1}}}(w)=\left(\frac{1-|w|^2}{1-\bar{w}\psi_{\beta, 1}(w)}\right)^{\gamma+2}= \left(\frac{(1-|w|^2)(1-\bar{\beta}w)}{1-\bar{\beta}w-|w|^2+\bar{w}\beta}\right)^{\gamma+2},
	\end{equation}
	where $\gamma>-1$.
	
	\begin{lemma}\label{lm3}
		On $A^2_{\gamma}(\mathbb{D}),$ for $w\in\mathbb D$ the real and imaginary parts of $\widetilde{C_{\psi_{\beta, 1}}}(w)$ are given by
		\begin{eqnarray*}
			&&\Re\{\widetilde{C_{\psi_{\beta, 1}}}(w)\}\\
			&&=c_{\beta, w}^{\gamma+2}\Big(((1-|w|^2)(1-\Re\{\bar{\beta}w\})+2\Im^2\{\bar{\beta}w\})^2+\Im^2\{\bar{\beta}w\}(1+|w|^2-2\Re\{\bar{\beta w}\})^2\Big)^{\frac{\gamma+2}{2}}\\
			&&\cos \left((\gamma+2)\tan^{-1}\left(\frac{\Im\{\bar{\beta}w\}(1+|w|^2-2\Re\{\bar{\beta w}\})}{(1-|w|^2)(1-\Re\{\bar{\beta}w\})+2\Im^2\{\bar{\beta}w\}}\right)\right)
		\end{eqnarray*}
		and 
		\begin{eqnarray*}
			&&\Im\{\widetilde{C_{\psi_{\beta, 1}}}(w)\}\\
			&&=c_{\beta, w}^{\gamma+2}\Big(((1-|w|^2)(1-\Re\{\bar{\beta}w\})+2\Im^2\{\bar{\beta}w\})^2+\Im^2\{\bar{\beta}w\}(1+|w|^2-2\Re\{\bar{\beta w}\})^2\Big)^{\frac{\gamma+2}{2}}\\
			&& \sin \left((\gamma+2)\tan^{-1}\left(\frac{\Im\{\bar{\beta}w\}(1+|w|^2-2\Re\{\bar{\beta w}\})}{(1-|w|^2)(1-\Re\{\bar{\beta}w\})+2\Im^2\{\bar{\beta}w\}}\right)\right),
		\end{eqnarray*}
		where $c_{\beta, w}=\frac{1-|w|^2}{|1-|w|^2+2i\Im \{\beta \bar{w}\}|^2}$.
	\end{lemma}

	\begin{proof}
		
		From \eqref{eq3}, we have
		
		\begin{eqnarray*}
			\widetilde{C_{\psi_{\beta, 1}}}(w)
			&&=\left(\frac{(1-|w|^2)(1-\bar{\beta}w)}{1-\bar{\beta}w-|w|^2+\bar{w}\beta}\right)^{\gamma+2}\\
			&&= \left(\frac{(1-|w|^2)(1-\bar{\beta}w)}{1-|w|^2+2i \Im\{\bar{w}\beta\}}\right)^{\gamma+2}\\
			&&=\left(\frac{(1-|w|^2)(1-\bar{\beta}w)(1-|w|^2-2i \Im\{\bar{w}\beta\})}{|1-|w|^2+2i \Im\{\bar{w}\beta\}|^2}\right)^{\gamma+2}.
		\end{eqnarray*}
		
		Let $c_{\beta, w}=\frac{1-|w|^2}{|1-|w|^2+2i\Im \{\beta \bar{w}\}|^2}$. 
		Thus, 
		\begin{eqnarray}\label{eq_4}
			\widetilde{C_{\psi_{\beta, 1}}}(w)
			&&=c_{\beta, w}^{\gamma+2}\left((1-\bar{\beta}w)(1-|w|^2-2i \Im\{\bar{w}\beta\})\right)^{\gamma+2}\nonumber\\
			&&=c_{\beta, w}^{\gamma+2}\Big((1-|w|^2)(1-\Re\{\bar{\beta}w\})+2\Im ^2(\bar{\beta}w)\nonumber\\
			&&+i \Im\{\bar{\beta}w\}(1+|w|^2-2\Re\{\bar{\beta}w\})\Big)^{\gamma+2}.
		\end{eqnarray}
		Now, consider 
		\begin{eqnarray*}
			(1-|w|^2)(1-\Re\{\bar{\beta}w\})+2\Im ^2(\bar{\beta}w)=r \cos \theta
		\end{eqnarray*}
		and 
		\begin{eqnarray*}
			\Im\{\bar{\beta}w\}(1+|w|^2-2\Re\{\bar{\beta}w\})=r \sin\theta,
		\end{eqnarray*}
		where $r=\left(((1-|w|^2)(1-\Re\{\bar{\beta}w\})+2\Im^2\{\bar{\beta}w\})^2+\Im^2\{\bar{\beta}w\}(1+|w|^2-2\Re\{\bar{\beta w}\})^2\right)^{\frac{1}{2}}$ and $-\frac{\pi}{2}<\theta<\frac{\pi}{2}.$
		
		Then it follows from \eqref{eq_4} that
		\begin{eqnarray*}
			\widetilde{C_{\psi_{\beta, 1}}}(w)
			&=&(rc_{\beta, w})^{\gamma+2}(\cos (\gamma+2)\theta+i\sin (\gamma+2)\theta).
		\end{eqnarray*}
		
		This gives the desired result.
	\end{proof}

    The following result shed light on certain geometric aspect of the Berezin range of $C_{\psi_{\beta, 1}}$ on $ A^2_{\gamma}(\mathbb{D}).$

	\begin{lemma}\label{lm_4}
		For the operator $C_{\psi_{\beta, 1}}$ on $ A^2_{\gamma}(\mathbb{D}),$ $\textbf{Ber}\{C_{\psi_{\beta, 1}}\}$ is closed under complex conjugation.
	\end{lemma}
	
	\begin{proof}
		Let $w=r e^{i\phi}$ and $\beta=\rho e^{i\theta}.$ We show that $\widetilde{C_{\psi_{\beta, 1}}}(r e^{i\phi})=\overline{\widetilde{C_{\psi_{\beta, 1}}}(r e^{i(2\theta-\phi)})},$ i.e.,
		\begin{eqnarray*}
			\left(\frac{1-r^2}{1-re^{-i\phi}\psi_{\beta, 1}(re^{i\phi})}\right)^{\gamma+2}=
			\left(\overline{\frac{1-r^2}{1-re^{-i(2\theta-\phi)}\psi_{\beta, 1}(re^{i(2\theta-\phi)})}}\right)^{\gamma+2}.
		\end{eqnarray*}
		This occurs if $e^{-i\phi}\psi_{\beta, 1}(re^{i\phi})=\overline{e^{-i(2\theta-\phi)}\psi_{\beta, 1}(re^{i(2\theta-\phi)})}.$ 
		Now,
		\begin{eqnarray*}
			e^{2i\theta} \overline{\psi_{\beta, 1}(re^{i(2\theta-\phi)})}&=&e^{2i\theta}
			\overline{\left(\frac{re^{i(2\theta-\phi)}-\beta}{1-\bar{\beta} re^{-i(2\theta-\phi)}}\right)}\\
			&&=e^{2i\theta}
			\frac{re^{i(\phi-2\theta)}-\rho e^{-i\theta}}{1-\rho e^{i\theta}re^{i(\phi-2\theta)}}\\
			&&=\frac{re^{i\phi}-\rho e^{i\theta}}{1-\rho e^{-i\phi}r e^{i\phi} }\\
			&&=\psi_{\beta, 1}(re^{i\phi}).
		\end{eqnarray*}
		
		This completes the proof.

	\end{proof}
	
With the previously established results, we are positioned to characterize convexity.
	\begin{theorem}\label{th_6}
		Let $C_{\psi_{\beta, 1}}\in \mathcal B(A^2_{\gamma}(\mathbb{D}))$ be such that $\Im\{\widetilde{C_{\psi_{\beta, 1}}}(w)\}=0$ implies $\Im\{\bar{\beta}w\}=0$. Then $\textbf{Ber}(C_{\psi_{\beta, 1}})$ is convex if and only if $\beta=0$.
	\end{theorem}
	
	\begin{proof}
		If $\beta=0$, then $\textbf{Ber}(C_{\psi_{\beta, 1}})=\{1\},$ which is convex. Conversely, suppose that $\textbf{Ber}(C_{\psi_{\beta, 1}})$ is convex. It follows from Lemma \ref{lm_4} that 
		$$\frac 12 \widetilde{C_{\psi_{\beta, 1}}}(w)+\frac 12 \overline{\widetilde{C_{\psi_{\beta, 1}}}(w)} =\Re\{\widetilde{C_{\psi_{\beta, 1}}}(w)\}\in\textbf{Ber}(C_{\psi_{\beta, 1}}).$$
		Then for each $w\in\mathbb D,$ we get $z\in\mathbb D$ such that $\widetilde{C_{\psi_{\beta, 1}}}(z)=\Re\{\widetilde{C_{\psi_{\beta, 1}}}(w)\}.$ Then $\Im\{\widetilde{C_{\psi_{\beta, 1}}}(z)\}=0.$ It implies that $\Im\{\bar{\beta}z\}=0$ and so $z=r\beta$ for some $r\in\left(-\frac{1}{|\beta|}, \frac{1}{|\beta|}\right)$. Thus we have 
		\begin{align*}
			&\widetilde{C_{\psi_{\beta, 1}}}(z)\\&=\Re\{\widetilde{C_{\psi_{\beta, 1}}}(r\beta)\}\\
			&=\left(\frac{1-|r\beta|^2}{|1-|r\beta|^2+2i\Im \{\beta \bar{r\beta}\}|^2}\right)^{\gamma+2}\\
			&[((1-|r\beta|^2)(1-\Re\{\bar{\beta}r\beta\})+2\Im^2\{\bar{\beta}r\beta\})^2+\Im^2\{\bar{\beta}r\beta\}(1+|r\beta|^2-2\Re\{\bar{\beta r\beta}\})^2]^{\frac{\gamma+2}{2}}\\
			&\cos \left((\gamma+2)\tan^{-1}\left(\frac{\Im\{\bar{\beta}r\beta\}(1+|r\beta|^2-2\Re\{\bar{\beta r\beta}\})}{(1-|r\beta|^2)(1-\Re\{\bar{\beta}r\beta\})+2\Im^2\{\bar{\beta}r\beta\}}\right)\right)\\
                &=\left(\frac{1}{1-|r\beta|^2}\right)^{\gamma+2}\left((1-|r\beta|^2)^2(1-r|\beta|^2)^2\right)^{\frac{\gamma+2}{2}} \cos((\gamma+2)\tan^{-1}0)\\
			&=(1-r|\beta|^2)^{\gamma+2}.
		\end{align*}
		Therefore, $\{\widetilde{C_{\psi_{\beta, 1}}}(r\beta):r\in(-\frac{1}{|\beta|}, \frac{1}{|\beta|})\}=((1-|\beta|)^{\gamma+2}, (1+|\beta|)^{\gamma+2}).$ Let $w=\rho e^{i\phi}.$ Then 
        $\lim_{\rho\to 1^-}\widetilde{C_{\psi_{\beta, 1}}}(\rho e^{i\phi})=0$ for $\beta\neq0$ and $1$ for $\beta=0.$ It follows that when $\beta\neq0,$ given $\epsilon$ with $0<\epsilon<(1-|\beta|)^{\gamma+2},$ there exists a point $w$ such that  $|\Re\{\widetilde{C_{\psi_{\beta, 1}}}(w)\}|<\epsilon.$ Now, for this $w$ there exists $z\in\mathbb D$ such that $|\widetilde{C_{\psi_{\beta, 1}}}(z)|=|\Re\{\widetilde{C_{\psi_{\beta, 1}}}(w)\}|<\epsilon.$ This contradicts the fact that $\widetilde{C_{\psi_{\beta, 1}}}(z)\in((1-|\beta|)^{\gamma+2}, (1+|\beta|)^{\gamma+2}).$ Hence $\beta=0.$
	\end{proof}

    The previous theorem considers the operator $C_{\psi_{\beta, 1}}$ on the space $A^2_{\gamma}(\mathbb{D}),$ under the assumption that $\Im\{\widetilde{C_{\psi_{\beta, 1}}}(w)\}=0$ implies $\Im\{\bar{\beta}w\}=0$. Below, we provide a range of values for $\gamma$ for which this operator inherently satisfies the stated condition.
	\begin{lemma}\label{lm_5}
		Let $-1<\gamma\leq 0$ and $w\in\mathbb{D}.$ Then $\Im\{\widetilde{C_{\psi_{\beta, 1}}}(w)\}=0$ if and only if $\Im\{\bar{\beta}w\}=0$.
	\end{lemma}
	\begin{proof}
		From Lemma \ref{lm3}, we obtain
		\begin{eqnarray}\label{eq_9}
			&&\Im\{\widetilde{C_{\psi_{\beta, 1}}}(w)\}\nonumber\\
			&&=c_{\beta, w}^{\gamma+2}\Big(((1-|w|^2)(1-\Re\{\bar{\beta}w\})+2\Im^2\{\bar{\beta}w\})^2+\Im^2\{\bar{\beta}w\}(1+|w|^2-2\Re\{\bar{\beta w}\})^2\Big)^{\frac{\gamma+2}{2}}\nonumber\\
			&& \sin \left((\gamma+2)\tan^{-1}\left(\frac{\Im\{\bar{\beta}w\}(1+|w|^2-2\Re\{\bar{\beta w}\})}{(1-|w|^2)(1-\Re\{\bar{\beta}w\})+2\Im^2\{\bar{\beta}w\}}\right)\right),
		\end{eqnarray}
		where $c_{\beta, w}=\frac{1-|w|^2}{|1-|w|^2+2i\Im \{\beta \bar{w}\}|^2}$.\\
		If  $\Im\{\bar{\beta}w\}=0$ then $\Im\{\widetilde{C_{\psi_{\beta, 1}}}(w)\}=0$. Conversely, let $\Im\{\widetilde{C_{\psi_{\beta, 1}}}(w)\}=0$. Clearly, $c_{\beta, w}>0, (1-|w|^2)(1-\Re\{\bar{\beta}w\})+2\Im^2\{\bar{\beta}w\}>0$ and $1+|w|^2-2\Re\{\bar{\beta w}\}>0$. Then it follows from \eqref{eq_9} that 
		$$\sin \left((\gamma+2)\tan^{-1}\left(\frac{\Im\{\bar{\beta}w\}(1+|w|^2-2\Re\{\bar{\beta w}\})}{(1-|w|^2)(1-\Re\{\bar{\beta}w\})+2\Im^2\{\bar{\beta}w\}}\right)\right)=0.$$ 
		Since $(1-|w|^2)(1-\Re\{\bar{\beta}w\})+2\Im^2\{\bar{\beta}w\}>0,$ it follows that
		$$-\frac{\pi}{2}<\tan^{-1}\left(\frac{\Im\{\bar{\beta}w\}(1+|w|^2-2\Re\{\bar{\beta w}\})}{(1-|w|^2)(1-\Re\{\bar{\beta}w\})+2\Im^2\{\bar{\beta}w\}}\right)<\frac{\pi}{2}.$$ \
		For $-1<\gamma\leq 0$, 
		$$-\pi<(\gamma+2)\tan^{-1}\left(\frac{\Im\{\bar{\beta}w\}(1+|w|^2-2\Re\{\bar{\beta w}\})}{(1-|w|^2)(1-\Re\{\bar{\beta}w\})+2\Im^2\{\bar{\beta}w\}}\right)<\pi.$$
		Thus, $\tan^{-1}\left(\frac{\Im\{\bar{\beta}w\}(1+|w|^2-2\Re\{\bar{\beta w}\})}{(1-|w|^2)(1-\Re\{\bar{\beta}w\})+2\Im^2\{\bar{\beta}w\}}\right)=0.$ This implies that 
		$\Im\{\bar{\beta}w\}=0$.
	\end{proof}

	The following corollary follows immediately from the lemma \ref{lm_5} and Theorem \ref{th_6}. 
	\begin{cor}
		Let $-1<\gamma\leq 0.$ For the operator $C_{\psi_{\beta, 1}}$ on $A^2_{\gamma}(\mathbb{D}),$ $\textbf{Ber}(C_{\psi_{\beta, 1}})$ is convex if and only if $\beta=0$.
	\end{cor}

	We conclude by establishing that $0$ lies within the closure of the Berezin range of $C_{\psi}$ , yet it is not contained in the Berezin range itself.
	
	\begin{proposition}\label{prop_0}
		Let $\psi$ be a non-identity holomorphic self-map on $\mathbb{D}$, then $0 \in \overline{\textbf{Ber}(C_{\psi})}\setminus \textbf{Ber}(C_{\psi}).$
	\end{proposition}

	\begin{proof}
		For any $w \in \mathbb{D},$ from \eqref{e1} we have 
		\begin{align}\label{e11}
			\widetilde{C_{\psi}}(w)=\left(\frac{1-|w|^2}{1-\bar{w}\psi(w)}\right)^{\gamma+2}.
		\end{align}
		Since $\psi$ is not an identity map on $\mathbb{D}$, there exists $w_0 \in \mathbb{T}$ such that the radial limit of $\psi$ exists at $w_0$ with $\psi(w_0) \neq w_0.$ Thus if $w$ radially approaches the point $w_0,$ then from \eqref{e11}, we get $\widetilde{C_{\psi}}(w) \to 0.$ Therefore, we obtain $0 \in \overline{\textbf{Ber}(C_{\psi})}.$
		
		Again, from \eqref{e11} we get
		\begin{align*}
			|\widetilde{C_{\psi}}(w)|=\frac{(1-|w|^2)^{\gamma+2}}{|1-\bar{w}\psi(w)|^{\gamma+2}} \geq \frac{1}{2^{\gamma+2}}(1-|w|^2)^{\gamma+2}>0\,\,\text{for all $w \in \mathbb{D}$}.
		\end{align*}
		Therefore, we obtain $\widetilde{C_{\psi}}(w) \neq 0$ for all $w \in \mathbb{D}.$ This completes the proof.
	\end{proof}

	\begin{remark}
		If $\psi$ is a non-identity holomorphic self-map on $\mathbb{D}$ then $\textbf{Ber}(C_{\psi})$ is not closed.
	\end{remark}

From Proposition \ref{prop_9} and Proposition \ref{prop_0}, we observe that if $C_{\psi,\phi }$ is a Weyl-type unitary operator or an unweighted non-identity composition operator then $0 \in \overline{\textbf{Ber}(C_{\psi,\phi })}\setminus \textbf{Ber}(C_{\psi,\phi }).$ However, if $\psi$ is such that $0\in\psi(\mathbb D)$ with $C_{\psi,\phi }\in \mathcal B(A^2_{\gamma}(\mathbb{D}))$ then $0\in\textbf{Ber}(C_{\psi,\phi })$. This observation motivates the following questions for future investigation regarding the inclusion of the origin in the Berezin range.

\noindent
\textbf{Question:}
Let $C_{\psi,\phi }$ be a bounded weighted composition operator on $A^2_{\gamma}(\mathbb{D}),$ then\\
$(i)$~~characterize $\psi$ and $\phi$ such that $0 \in  \overline{\textbf{Ber}(C_{\psi,\phi })} \setminus \textbf{Ber}(C_{\psi,\phi })$.\\
$(ii)$~~characterize $\psi$ and $\phi$ such that $0 \in \textbf{Ber}(C_{\psi,\phi })$.\\
$(iii)$~~characterize $\psi$ and $\phi$ such that $0$ is contained in the interior of $\textbf{Ber}(C_{\psi,\phi })$.

Notably, in \cite{BS_IEOT_2002} Burdon and Shapiro provide important insights into precisely when the origin lies in the numerical range of composition operators on the Hardy-Hilbert space, which may offer valuable guidance for exploring similar questions about the Berezin range of weighted composition operators on both the Hardy-Hilbert space and weighted Bergman spaces.

		\section*{Declarations}

\noindent\textit{Conflict of interest.} The authors declare that data sharing is not applicable to this article, as no data sets were generated or analyzed during the current study. The authors declare no conflict of interest.

\noindent\textit{Acknowledgements.} Miss Somdatta Barik would like to thank UGC, Govt. of India, for the financial support in the form of Senior Research Fellowship under the mentorship of Prof. Kallol Paul.

	\bibliographystyle{amsplain}
	
\end{document}